\documentclass[11pt]{amsart}
\usepackage{amssymb,mathrsfs,enumitem,amsmath,bbold}
\usepackage[mathscr]{euscript}
\usepackage{color}

\usepackage[all]{xy}
\definecolor{Chocolat}{rgb}{0.36, 0.2, 0.09}
\definecolor{BleuTresFonce}{rgb}{0.215, 0.215, 0.36}
\definecolor{EgyptianBlue}{rgb}{0.06, 0.2, 0.65}

\usepackage[colorlinks,final,hyperindex]{hyperref}
\hypersetup{citecolor=BleuTresFonce, urlcolor=EgyptianBlue, linkcolor=Chocolat}

\usepackage{fourier}

\newtheorem{theorem}{Theorem}[section]
\newtheorem{corollary}{Corollary}[theorem]

\newtheorem{lemma}{Lemma}[theorem]
\newtheorem{proposition}[theorem]{Proposition}

\newtheorem{conjecture}[theorem]{Conjecture}

\theoremstyle{definition}
\newtheorem{example}[theorem]{Example}

\newtheorem{definition}[theorem]{Definition}

\DeclareMathAlphabet{\pazocal}{OMS}{zplm}{m}{n}

\def\calG{\pazocal{G}}

\def\calL{\pazocal{L}}
\def\calM{\pazocal{M}}

\def\calT{\pazocal{T}}

\DeclareMathOperator{\HyperCom}{HyperCom}
\DeclareMathOperator{\Grav}{Grav}

\DeclareMathAlphabet{\mathbbold}{U}{bbold}{m}{n}

\def\k{\mathbbold{k}}

\begin{document}

\title{Homotopy invariants for $\overline{\calM}_{0,n}$ via Koszul duality}

\author{Vladimir Dotsenko}

\address{Institut de Recherche Math\'ematique Avanc\'ee, UMR 7501, Universit\'e de Strasbourg et CNRS, 7 rue Ren\'e-Descartes, 67000 Strasbourg CEDEX, France}

\email{vdotsenko@unistra.fr}


\begin{abstract}
We show that the integer cohomology rings of the moduli spaces of stable rational marked curves are Koszul. This answers an open question of Manin. Using the machinery of Koszul spaces developed by Berglund, we compute the rational homotopy Lie algebras of those spaces, and obtain some estimates for Betti numbers of their free loop spaces in case of torsion coefficients. We also prove and conjecture some generalisations of our main result.

\end{abstract}

\maketitle


\section{Introduction}

The intersection theory on $\overline{\calM}_{0,n}$ --- the Deligne--Mumford compactification of moduli spaces of complex projective lines with $n$ marked points --- is omnipresent, arising in a meaningful way in areas ranging from combinatorics and enumerative algebraic geometry to integrable hierarchies and Batalin--Vilkovisky formalism. Even though the Chow rings of those spaces can be explicitly presented by generators and relations and admit various nice bases, not all natural questions about those rings admit easy answers. In particular, one may identify Chow rings with cohomology, and try to use them to determine various homotopy invariants for the Deligne--Mumford compactifications. For that, it would be highly advantageous to know that the rational cohomology algebras of $\overline{\calM}_{0,n}$ are Koszul. This question has been open for some 15 years; Manin~\cite[Section 3.6.3]{Man} asked it in a more general context of genus zero components of the extended modular operad~\cite{LMExt}; Readdy~\cite{Readdy} mentioned that the same had been asked by Reiner, and Petersen~\cite{Pet} asked this a few years after that. 

Let us summarise the state-of-the-art in this problem. The positive answer has been known for stable rational curves with at most six marked points. In the cases of three and four marked points, the problem is trivial, for five marked points, a solution is a one-line exercise in Koszul algebras, the case of six marked points is claimed to have been resolved in the preprint \cite{Iyu} that has been in circulation for several years. That preprint suggests a rather drastic difference already between the cases of five and six points: proof of Koszulness for the latter relies on an interpretation of the Koszul dual algebra as a potential algebra and has, as an intermediate step, a computation of a Gr\"obner basis modulo a prime number $p=7$ for Hilbert series estimates allowing one to compare the latter potential algebra with a simpler one; generalising that argument for a higher number of points does not seem realistic.

The main result of this paper is that the integer cohomology ring $H^\bullet(\overline{\calM}_{0,n},\mathbb{Z})$ is Koszul for all~$n$ (Theorem \ref{th:Koszul}). To that end, we show that these rings have (commutative) quadratic Gr\"obner bases, which should appear very surprising to anyone who spent some time working on this problem. Although exhibiting a quadratic Gr\"obner basis of defining relations for a given algebra is the ``cheapest'' way to prove that the algebra is Koszul, until now, there has been no success with this strategy for cohomology of the Deligne--Mumford compactifications. There is a very good reason for that: the two standard presentations of the cohomology algebra (due to Keel \cite{Keel} and De Concini--Procesi \cite{DCP}) can be shown to not have a quadratic Gr\"obner basis for \emph{any} ordering of monomials, see Proposition \ref{prop:KeelNotPBW}. We note that non-quadratic Gr\"obner bases for those rings are well known, see, e.g. work of Yuzvinsky~\cite{Yuz}. 

The presentation of cohomology that we use was discovered independently in work of Etingof, Henriques, Kamnitzer and Rains~\cite[Proposition 5.3]{EHKR} and in the Ph.~D. thesis of Singh~\cite[Section 9.2]{Singh} (of which we learned from Strickland's comment on the abovementioned Petersen's question); first attempts to work with those presentations also led to non-quadratic Gr\"obner bases \cite[Theorem 5.5]{EHKR}. We discovered somewhat exotic admissible orderings of monomials for which these rings have quadratic Gr\"obner bases; our inspiration came from a a certain monomial basis for the operad $\HyperCom$ consisting, component-wise, of homology $H^\bullet(\overline{\calM}_{0,n+1},\mathbb{Z})$. The operad $\HyperCom$ was introduced by Getzler \cite{Get} and studied by several authors; our monomial basis for it is related to our previous work with Khoroshkin \cite{DKQ} in which we proposed a quadratic Gr\"obner basis of relations for that operad. To the best of our knowledge, in general there is no formal relationship between Gr\"obner bases on the operadic and the algebraic side.

Our results have immediate applications to otherwise hard questions of algebraic topology. Using the machinery of Koszul spaces \cite{Ber}, we were able to compute explicit presentations of the rational homotopy Lie algebras of $\overline{\calM}_{0,n}$ by generators and relations and to give formulas for generating series of ranks of rational homotopy groups of these spaces (Theorem \ref{th:RHT}). The latter calculation is similar in spirit to that of Kohno \cite{Ko} and Falk--Randell \cite{FR}. We also use our result to estimate the Betti numbers of the free loop spaces $L\overline{\calM}_{0,n}$ with torsion coefficients (Theorem \ref{th:FreeLoopTorsion}; same method works in characteristic zero, but it simply recovers the known result of Lambrechts~\cite{Lam} for the Betti numbers of the free loop space of a $\mathbb{Q}$-coformal space). 

It is natural to ask for extensions of our result. In \cite[Remark 2.17]{BBhighly}, a question is raised of a geometric characterisation of the class of K\"ahler manifolds with Koszul cohomology algebras; such manifolds are of interest since they are coformal (over rational numbers), and their rational homotopy Lie algebras can be described in a very explicit way. We show that a smooth projective toric variety belongs to that class if and only if its fan is a flag complex (Theorem \ref{th:ToricKoszul}); in particular, this includes the Losev--Manin spaces~\cite{LM} and the noncommutative Deligne--Mumford spaces~\cite{DSV}. We conjecture that certain types of De Concini--Procesi wonderful models of hyperplane arrangements \cite{DCP} also belong to that class; another conjectural class of Koszul algebras of the same flavour is given by matroid Chow rings~\cite{AHK}. 

Our work prompts one more natural question. The collection of all moduli spaces of stable rational marked curves forms a topological operad, and therefore its homology is an operad in the category of quadratic coalgebras. Recently, Manin \cite{Man18} proposed a beautiful idea that the Koszul duality for cohomology algebras of moduli spaces of stable rational curves is meaningful in the context of certain quantised actions of its homology operad $\HyperCom$ (encoding the genus zero part of a cohomological field theory). It would be interesting to know how our Koszulness result fits into this context; foundations for the framework where this question belongs have recently been developed by Manin and Vallette~\cite[Section 5]{MV19}. In particular, a suggestion of Bruno Vallette is that understanding rational homotopy theory of the operad $\left\{\overline{\calM}_{0,\bullet+1}\right\}$ viewed as a Hopf operad is relevant for this purpose. A supporting argument towards this suggestion emerged after the first version of this paper was circulated. Namely, one may replace the spaces $\overline{\calM}_{0,n+1}$ by their real versions (studied in some detail in \cite{EHKR}); in that case, cohomology algebras are known to be generated by elements of homological degree one subject to quadratic relations, and are somewhat similar in flavour to the cohomology algebras of complements of complex hyperplane arrangements studied, for instance in \cite{FR,Ko}. In a recent preprint \cite{KW}, Khoroshkin and Willwacher answered the question posed in \cite{EHKR}, proving that the rational cohomology algebras of $\overline{\calM}_{0,n+1}(\mathbb{R})$ are Koszul; their proof uses a Hopf operad model for the homology operad related to the theory of Kontsevich's graph complexes \cite{Wi}. While there is no direct relationship between the real and the complex case, the fact that in both cases the \emph{operad} theory is used to establish Koszulness of \emph{algebras} is remarkable and deserves further attention.    

\subsection*{Acknowledments. } I am grateful to Yuri Ivanovich Manin for encouragement. I~also wish to thank Greg Arone, Alexander Berglund, Alex Fink, Vincent G\'elinas, Anton Khoroshkin, Natalia Iyudu, Vic Reiner, Pedro Tamaroff, and Bruno Vallette for useful discussions. The paper benefited from extraordinarily thorough peer review process, and I wish to offer my deep gratitude to the anonymous referee whose queries made me spell out the proof in great detail, leaving no stone unturned. Special thanks are due to Neil Strickland for providing a copy of~\cite{Singh}, and to Geoffroy Horel for a discussion of results of~\cite{CH}.

Let us introduce two notational choices used throughout the paper. We shall benefit from a viewpoint that makes one of the marked points on a stable curve distinguished (similarly to how it is done when moduli spaces are used in the operad theory~\cite{Get}); to highlight that, we choose the parameter $n$ for numbering the moduli spaces so that the $n$-th space is $\overline{\calM}_{0,n+1}$. We also frequently use the ``topologist's notation'' $\underline{n}=\{1,\ldots,n\}$ when working with finite sets. Throughout the paper $\k$ denotes an arbitrary unital commutative ring. 

\section{Recollections}

In this section, we briefly recall various presentations of the integral cohomology rings of moduli spaces of stable rational curves via generators and relations, and give a short survey of Gr\"obner bases for algebras and operads, including both the use of Gr\"obner bases in the theory of Koszul algebras and the application of operadic Gr\"obner basis to exhibiting an explicit combinatorial basis of the homology of $\overline{\calM}_{0,n+1}$. We discuss operads only to the extent necessary to arrive at a basis of $H_{\bullet}(\overline{\calM}_{0,n+1},\mathbb{Z})$ used in our main argument; the interested reader is invited to consult the monograph \cite{LV} for a state-of-the-art introduction to the theory of algebraic operads, and the monograph \cite{BD} for details on shuffle operads and operadic Gr\"obner bases.

\subsection{Presentations of the cohomology ring}

\subsubsection{Keel presentation}
The most famous presentation of the cohomology ring of $\overline{\calM}_{0,n+1}$ goes back to Keel~\cite{Keel} who computed the Chow rings 
 \[
\mathbf{A}^\bullet(\overline{\calM}_{0,n+1})\cong H^{2\bullet}(\overline{\calM}_{0,n+1},\mathbb{Z}).
 \] 
In particular, he gave an explicit presentation of the cohomology ring of $\overline{\calM}_{0,n+1}$ via generators and relations; this presentation is at heart of most subsequent results that use intersection theory of $\overline{\calM}_{0,n+1}$. According to \cite[Claim (6), p.~550]{Keel}, the cohomology $H^{\bullet}(\overline{\calM}_{0,n+1},\mathbb{Z}$ is generated by elements $D_S$ with $S\subset\underline{n+1}$, $2\le |S|\le n-1$ subject to the relations
\begin{itemize}
\item[\textbullet] $D_S=D_{\underline{n+1}\setminus S}$ for all $S$,
\item[\textbullet] $\sum\limits_{i,j\in S, k,l\notin S}D_S=\sum\limits_{i,k\in S, j,l\notin S}D_S=\sum\limits_{i,l\in S, j,k\notin S}D_S$ for all pairwise distinct $i,j,k,l$,
\item[\textbullet] $D_SD_T=0$ for all $S,T$ with $S\cap T\ne\varnothing$, $S\not\subseteq T$, $T\not\subseteq S$. 
\end{itemize}
Geometrically, the class $D_S$ correspond to the divisor $D^S$ whose generic element is the curve with two components, the points of $S$ on one branch, the points of the complement $\underline{n+1}\setminus I$ on the other.

\subsubsection{De Concini--Procesi presentation}
Another presentation of the cohomology ring is due to De Concini and Procesi \cite{DCP}, who managed to interpret $\overline{\calM}_{0,n+1}$ as the wonderful model of the Coxeter arrangement of type $A_{n-1}$ for the minimal building set of the corresponding lattice of subspaces. More precisely, one should consider, for each $I\subseteq\underline{n}$, the vector space $V_I:=\mathbb{C}^I/(1,1,\ldots,1)$. There is an obvious well defined map from $\mathbb{C}^n\setminus\bigcup_{i\ne j}\{x_i=x_j\}$ to the product 
 \[
\prod_{I\subseteq\underline{n}, |I|\ge 3} \mathbb{P}(V_I) ,
 \]
and the closure of the image of this map is isomorphic to $\overline{\calM}_{0,n+1}$~\cite{DCP}. The resulting presentation of the cohomology ring $H^\bullet(\overline{\calM}_{0,n+1},\mathbb{Z})$ is as follows: the generators are $Y_S$ with $S\subseteq\underline{n}$, $2\le |S|\le n$, and the relations between them are 
\begin{itemize}
\item[\textbullet] $\sum\limits_{i,j\in S} Y_S=0$ for all $i\ne j$,
\item[\textbullet] $Y_SY_T=0$ for all $S,T$ with $S\cap T\ne\varnothing$, $S\not\subseteq T$, $T\not\subseteq S$. 
\end{itemize}
It is easy to check that if one makes the marked point $n+1$ in the Keel presentation distinguished and eliminates $D_S$ for $n+1\in S$ using the symmetry relation $D_S-D_{S^C}$, and also eliminates $Y_{\underline{n}}$ from the De Concini--Procesi presentation using the linear relation for $S=\underline{n}$, the two presentations become identical. Thus, most of the generators $Y_T$ coincide with the corresponding Keel generators $D_T$ and arise from boundary divisors; the only exception is the class $Y_{\underline{n}}$ which can be seen to be equal to the minus the first Chern class of the tautological line bundle at the $0$th marked point. Generalising this geometric intuition leads to a more convenient presentation which we shall now describe.

\subsubsection{Etingof--Henriques--Kamnitzer--Rains--Singh presentation}
The least known presentation of the cohomology ring was found independently in the Ph. D. thesis of Singh~\cite[Section 9.2]{Singh} and in work of Etingof, Henriques, Kamnitzer and Rains~\cite[Proposition 5.3]{EHKR}  who utilised it for computation of mod~$2$ cohomology of the real locus of~$\overline{\calM}_{0,n+1}$. This presentation identifies $H^\bullet(\overline{\calM}_{0,n+1},\mathbb{Z})$ with the ring $\mathbf{R}_n$, the quotient of the polynomial ring in variables $X_S$, $S\subseteq\underline{n}$, $|S|\ge 3$, modulo the ideal~$\mathbf{I}_n$ generated by the following three groups of polynomials:
\begin{itemize}
\item[\textbullet] $X_S^2$, where $|S|=3$,
\item[\textbullet] $X_S(X_S-X_{S\setminus\{s\}})$, where $|S|>3$, and $s\in S$,
\item[\textbullet] $(X_{S\cup T}-X_S)(X_{S\cup T}-X_T)$, where $S\cap T\ne\varnothing$, $S\not\subseteq T$, $T\not\subseteq S$. 
\end{itemize}
This presentation of the cohomology ring has its own geometric interpretation. It is established in \cite{Singh} that an isomorphism from $\mathbf{R}_n$ onto the cohomology ring can be implemented by sending $X_I$ to $\pi_I^*(c_1(\calL_I))$, where 
 \[
\pi_I\colon \overline{\calM}_{0,n+1}\to\mathbb{P}(V_I)
 \]
is the projection map arising from the De Concini--Procesi construction, $\calL_I$ is the tautological line bundle on $\mathbb{P}(V_I)$, and $c_1$ is the first Chern class. Algebraically, the formula $X_S:=\sum_{S\subseteq T} Y_T$ implements the conversion from the De Concini--Procesi presentation to this one; as a result, linear relations become redundant generators and can be eliminated directly, so that the passage from the Keel presentation to the De Concini--Procesi presentation and to the Etingof--Henriques--Kamnitzer--Rains--
Singh presentation is a gradual transformation of the set of the generating elements to the minimal one. To accommodate inductive arguments, we shall consider the ring $\mathbf{R}_n$ also for $n=1$, in which case it is simply $\mathbb{Z}$. (There are no stable curves with $2=1+1$ marked points, but the above purely algebraic definition of the ring $\mathbf{R}_n$ does not know that.) 

\subsection{Gr\"obner methods}

\subsubsection{Gr\"obner bases for algebras and operads}

In this article, we use two different kinds of Gr\"obner bases: for commutative rings and for operads, so we feel that it is reasonable to offer the reader a unified view of the theory of Gr\"obner bases. For us, the theory of Gr\"obner bases emerges in situations where one deals with algebraic structures of a particular kind: one assumes that the free $\k$-algebra generated by a set $X$ has a $\k$-basis of ``monomials'' obtained by evaluating a certain combinatorial species on~$X$, such that the result of applying any structure operation to monomials is again a monomial. Some key examples of such situations are given by commutative associative rings (the free algebra generated by $X$ is $\{$is $\mathbb{Z}[X]$, and the product of any monomials is a monomial), associative rings (the free algebra generated by $X$ is the tensor algebra $T(X)$, and the concatenation of two noncommutative monomials is a noncommutative monomial), and shuffle operads (the free algebra generated by $X$ is the free shuffle operad $\calT_{\mathrm{sh}}(X)$, and any shuffle composition of several shuffle trees is a shuffle tree). An ordering of monomials in the free algebra is said to be \textsl{admissible} if it is a total well ordering, and if each structure operation is an increasing function of its arguments: replacing any one of the monomials to which one applies that structure operation by a greater one increases the result. 

\begin{definition}[Gr\"obner basis]
Given an admissible ordering of the free $\k$-algebra generated by $X$, one can define a \textsl{Gr\"obner basis} of an ideal $I$ of that free algebra as a subset $G\subset I$ for which the leading coefficient of every element of $G$ is equal to $1$, and the leading monomial of every element of $I$ is divisible by a leading monomial of an element of $G$. (The requirement on the leading coefficients is unimportant over a field, but is non-empty over a ring.) 
\end{definition}

The primary reason to look for Gr\"obner bases is dictated by considerations of linear algebra: a Gr\"obner basis for an ideal gives extensive information on the quotient modulo $I$. To state a precise result, we need one more definition.

\begin{definition}[Normal monomial]
A monomial is said to be \textsl{normal} with respect to $G$ if it is not divisible by any of the leading monomials of elements of $G$. 
\end{definition}

It is easy to show that the normal monomials with respect to any set of generators of an ideal $I$ always form a $\k$-spanning set of the quotient ring modulo $I$. The following result shows how to strengthen this property to obtain a Gr\"obner basis criterion.

\begin{proposition}\label{prop:GBCrit}
Suppose that $G$ is a generating set $G$ of an ideal $I$, and the leading coefficient of each element of $G$ equal to $1$. Then $G$ is a Gr\"obner basis (respectively, a noncommutative Gr\"obner basis) if and only if the quotient modulo $I$ is a free $\k$-module with a basis of cosets of monomials that are normal with respect to $G$.
\end{proposition}

\begin{proof}
This is established by an argument identical to \cite[Prop.~2.3.3.5]{BD}.\qed
\end{proof}

\subsubsection{Koszulness and Gr\"obner bases}\label{sec:GB}

This section offers a short survey of Koszul algebras, including the use of Gr\"obner bases in the theory of Koszul algebras. We also  discuss suitability of various presentations of $H^\bullet(\overline{\calM}_{0,n+1},\mathbb{Z})$ for the purpose of proving Koszulness of that algebra using Gr\"obner bases. For equivalent definitions and various properties of Koszul algebras, we refer the reader to~\cite{PP}. 

Let $A$ be a weight graded (commutative or noncommutative) associative $\k$-algebra. We assume the weight grading to be \emph{standard}, in other words, we assume $A$ to be generated by elements of weight~$1$ (in particular, $A_0=\k$). Such an algebra is automatically augmented, and $\k$ acquires a trivial $A$-module structure via the augmentation map.

\begin{definition}[Koszul algebra]
 The algebra $A$ is said to be \textsl{Koszul} if the trivial module has a resolution by free $A$-modules 
 \[
\cdots\to A^{a_n}\stackrel{d_n}{\to} A^{a_{n-1}}\stackrel{d_{n-1}}{\to}\cdots \stackrel{d_2}{\to} A^{a_1}\stackrel{d_1}{\to} A\stackrel{\epsilon}{\to} \k\to 0 ,
 \]
where the differentials $d_k$ are ``linear'', i.e. their matrices consist of elements of weight~$1$. 
\end{definition}

Over a field $\k$, it is well known that an algebra whose ideal of relations has a Gr\"obner basis consisting of quadratic elements (i.e. elements of weight~$2$) is Koszul. In fact, it is known that an associative algebra with a noncommutative quadratic Gr\"obner basis is Koszul \cite[Chapter 4, Theorem~3.1]{PP}, and also that a commutative associative algebra with a quadratic Gr\"obner basis is Koszul \cite[Chapter 4, Theorem~8.1]{PP}. It turns out that the same result is available over any unital commutative ring~$\k$. 

\begin{proposition}\label{prop:GBKoszul}
A (commutative or noncommutative) associative $\k$-algebra with a quadratic Gr\"obner basis is Koszul.
\end{proposition}

\begin{proof}
The algebraic discrete Morse theory proof works without any adjustments: see \cite[Corollary 3.9]{JW} in the commutative case and \cite[Corollary 4.9]{JW} for the noncommutative one. \qed
\end{proof}

We remark that the two Gr\"obner bases criteria (the commutative and the noncommutative one) are different: the universes where one looks for admissible orderings are of different nature, and converting one approach into another is not always possible; we refer the reader to \cite[p.~93]{PP} for insight into that. 

\subsubsection{Combinatorial basis of \texorpdfstring{$H_{\bullet}(\overline{\calM}_{0,n+1},\mathbb{Z})$}{H}}

In this section, we briefly recall how to obtain a particular combinatorial basis of the homology $H_{\bullet}(\overline{\calM}_{0,n+1},\mathbb{Z})$ that was first discovered in \cite[Sec.~5.4.3]{DKQ}. We start with recalling the Koszul dual pair of operads $\HyperCom$ and $\Grav$. The operad $\HyperCom$ is generated by $S_k$-invariant operations $\nu_k$ (one for each $k\ge 2$) of homological degree $2(k-2)$, satisfying a number of identities: for each $n$ and each triple of distinct numbers $(a,b,c)\in\underline{n}\times\underline{n}\times\underline{n}$, one has the relation
 \[
\sum
\vcenter{
\xymatrix@R=1mm@C=2mm{
a\ar@{-}[dr]&b\ar@{-}[d]&\hspace{-2.5mm}\ldots\ar@{..}[dl]&\ar@{-}[dll]\\
&*++[o][F-]{}\ar@{-}[dr]& \ldots\ar@{..}[d]&\ar@{-}[dl]c&\\
&&*++[o][F-]{}\ar@{-}[d]&&&&\\
&&
}} =
\quad \sum
\vcenter{
\xymatrix@R=1mm@C=2mm{
a\ar@{-}[dr]&c\ar@{-}[d]&\hspace{-2.5mm}\ldots\ar@{..}[dl]&\ar@{-}[dll]\\
&*++[o][F-]{}\ar@{-}[dr]& \ldots\ar@{..}[d]&\ar@{-}[dl]b&\\
&&*++[o][F-]{}\ar@{-}[d]&&&&\\
&&
}} .
 \]
Here each internal vertex of each of the trees corresponds to one of the operations $\nu_k$, and the sums are over all trees with two internal vertices and $n$ leaves, and all ways to distribute the labels from $\underline{n}\setminus\{a,b,c\}$ between the leaves. The operad $\Grav$ is generated by $S_k$-invariant operations $\lambda_k$ of homological degree~$1$ (one for each $k\ge 2$), satisfying a number of identities: for each $n$ and each subset $I\subset\underline{n}$ of cardinality at least $3$, one has the relation
 \[
\sum_{\{s_1,s_2\}\subset S}
\vcenter{
\xymatrix@R=1mm@C=2mm{
s_1\ar@{-}[dr]& &s_2\ar@{-}[dl]&\\
&*++[o][F-]{}\ar@{-}[dr]& \ldots\ar@{..}[d]&\ar@{-}[dl]&\\
&&*++[o][F-]{}\ar@{-}[d]&&&\\
&&
}}=
\begin{cases}
\vcenter{
\xymatrix@R=1mm@C=2mm{
\ar@{-}[dr]&S\ar@{-}[d]&\ar@{-}[dl]&\\
&*++[o][F-]{}\ar@{-}[dr]& \ldots\ar@{..}[d]&\ldots\ar@{-}[dl]&\\
&&*++[o][F-]{}\ar@{-}[d]&&&\\
&&
}}, \quad S\neq\underline{n},\\
\quad \quad\qquad 0\qquad \qquad \qquad\!\!\!, \quad  S=\underline{n}.
\end{cases}
 \]
where each internal vertex corresponds to one of the operations $\lambda_k$, and the tree on the right is the unique tree with two internal vertices and $n$ leaves for which the leaves not connected to root are labelled by the subset $S$. 

We shall now describe a combinatorial basis of the operad $\HyperCom$ which is of crucial importance for our proof of the main result of the paper. 

\begin{proposition}\label{prop:BasisHyperCom}
The operad $\HyperCom$ admits a combinatorial basis $\mathsf{B}$ of cosets of shuffle tree monomials defined inductively as follows:
\begin{itemize}
\item[\textbullet] the tree $\stackrel{a}{\mid}$ without internal vertices belongs to $\mathsf{B}$,
\item[\textbullet] a shuffle tree monomial $\vcenter{
\xymatrix@R=1mm@C=2mm{
\tau_1\ar@{-}[dr]&\cdots\ar@{-}[d]&\tau_k\ar@{-}[dl]\\
&*++[o][F-]{}\ar@{-}[d]&\\
&&
}}$ belongs to $\mathsf{B}$ if and only if $\tau_1$, \ldots, $\tau_k$ belong to $\mathsf{B}$, and the root vertices of $\tau_1$, \ldots, $\tau_{k-1}$ are not labelled with $\nu_2$. 
\end{itemize}
\end{proposition}

\begin{proof}
We begin with defining a particular admissible ordering of monomials in the free operad generated by operations $\lambda_k$. First, one  considers an exotic weight grading that assigns weight $1$ to each generator $\lambda_k$ with $k>2$, and weight $0$ to the generator $\lambda_2$. This leads to a partial order which compares exotic weights of monomials. This already ensures that for $|S|<n-1$, the monomial on the right hand side of the corresponding relation of $\Grav$ is the leading monomial of that relation. To proceed, one uses the theory of word operads \cite{Dot}, and considers the monoid of ``quantum monomials'' $\mathsf{QM}=\langle x,y,q\mid xq=qx,yq=qy,yx=xyq\rangle$ and the map from the free operad generated by all $\lambda_k$ to the word operad associated to $\mathsf{QM}$ sending $\lambda_2$ to $(y,y)$ and $\lambda_k$ to $(x,x,\ldots,x)$ for $k>2$. This leads to a partial order which compares the elements of the word operad associated to the given monomials. This ensures that for $|S|=n-1$, the monomial on the right hand side of the corresponding relation of $\Grav$ is the leading monomial of that relation. Finally, one considers the total order extension of the partial order defined so far using the reverse path-permutation lexicographic ordering \cite[Def.~5.4.1.8]{BD}, so that the leading term of the relation for $|S|=n$ is 
 \[
\vcenter
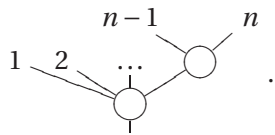
{
\xymatrix@R=1mm@C=2mm{
&&&n-1\ar@{-}[dr]& &n\ar@{-}[dl]\\
&1\ar@{-}[drr]&2\ar@{-}[dr]& \ldots\ar@{-}[d]&*++[o][F-]{}\ldots\ar@{-}[dl]\\
&&&*++[o][F-]{}\ar@{-}[d]&&\\
&&&
}} .
 \]
It is established in \cite[Th.~5.16]{DKQ}, that the number of monomials with $n$ leaves that are normal with respect to these leading monomials is equal to the dimension of $\Grav(n)$, so Proposition \ref{prop:GBCrit} ensures that for the ordering we just defined, the defining relations of the shuffle operad $\Grav$ form a Gr\"obner basis. By duality, the same is true for the operad $\HyperCom$ if one considers the opposite ordering for the corresponding free operad; the leading monomials of relations of $\Grav$ that we just described become precisely the normal monomials of weight two for the operad $\HyperCom$. Since we have a Gr\"obner basis, all normal monomials for $\HyperCom$ are the shuffle trees for which each divisor of weight two is normal. It remains to note that such monomials are described precisely by the recursive rule defining the set $\mathsf{B}$. \qed
\end{proof}

\section{Koszulness of integral cohomology}\label{sec:main}

In this section, we prove the main result of this paper. Our strategy is to exhibit a \emph{commutative} quadratic Gr\"obner bases (in fact, for the case of at least six marked points it is not clear if there exists a noncommutative quadratic Gr\"obner basis, even after a linear change of variables). Since all the cohomology classes considered in this paper are of even degree, we may forget about the sign rule normally pertinent when working with cohomology rings, and impose a weight grading for which the weight of an element is equal to one half of its cohomological degree. This brings us to the ``classical'' commutative algebras to which one normally applies the theory of Gr\"obner bases: quotients of polynomial algebras equipped with the standard weight grading discussed in Section \ref{sec:GB}.

We start with explaining why neither the Keel presentation nor the De Concini--Procesi presentation \emph{cannot} be used to prove the Koszulness of $H^\bullet(\overline{\calM}_{0,n+1},\mathbb{Z})$ by means of the Gr\"obner bases theory. Recall that both presentations have an excessive number of generators which are compensated for by the presence of linear relations, so a hypothetical Gr\"obner basis proving Koszulness for those presentations will have some linear elements (eliminating the excess) and some quadratic ones. 
We recall that a Gr\"obner basis $G$ is said to be \textsl{reduced} if all the leading monomials of its elements are pairwise distinct and all the non-leading monomials are normal with respect to $G$.

\begin{proposition}\label{prop:KeelNotPBW}
Neither the Keel presentation nor the De Concini--Procesi presentation for $H^\bullet(\overline{\calM}_{0,n+1},\mathbb{Z})$ admits a linear-and-quadratic Gr\"obner basis (either commutative or noncommutative) for $n\ge 4$, no matter what admissible ordering of generators one chooses. 
\end{proposition}

For the interested reader, we remark that by standard results on PBW-bases \cite[Chapter 4, Theorem 4.1]{PP}, computing noncommutative Gr\"obner bases for the Koszul dual algebras would not help, since an algebra has a quadratic Gr\"obner basis if and only if its Koszul dual does (for the opposite monomial ordering).

\begin{proof}
We start with a detailed argument for the Keel presentation. Suppose that there exists a linear-and-quadratic Gr\"obner basis, which we may assume reduced. Since $H^\bullet(\overline{\calM}_{0,n+1},\mathbb{Z})$ is an Abelian group of finite rank, the monomials $D_S^k$, $k\ge 1$, cannot be all linearly independent, so one of them must be divisible by a leading monomial of some Gr\"obner basis element. Thus, for each $S$, either $D_S$ or $D_S^2$ appears as the leading monomial in our Gr\"obner basis. Let us consider the smallest variable $D_S$ which does not appear as a leading monomial (i.e. the smallest variable that is not eliminated). This means that $D_S^2$ appears as a leading monomial; moreover, since the Gr\"obner basis we consider is assumed reduced, this means that $D_S^2$ is a relation, since there can be no lower terms in the corresponding Gr\"obner basis element. Consequently, $D_S^2=0$ in $H^\bullet(\overline{\calM}_{0,n+1},\mathbb{Z})$.  However, the variables $D_S$ correspond to appropriate divisors on $\overline{\calM}_{0,n+1}$ for which the self-intersection is never zero for $n\ge 4$ (for $n=2$, there are no classes $D_S$, and for $n=3$, there is just one, and it squares to zero); this contradiction completes the proof. For the De Concini--Procesi presentation, the only main difference in the above argument is that one should rule out the possible relation $Y_{\underline{n}}^2=0$. However, it is known \cite[Cor.~1]{FY} that the monomial $Y_{\underline{n}}^k$ does not vanish in the cohomology algebra for $k<n-1$. Since $n-1\ge 3$, the result follows. \qed 
\end{proof}

From this proof, we obtain a useful hint: a homogeneous quadratic Gr\"obner basis must contain a square of a variable, and, more generally, must have all squares of variables as leading monomials, leading to a square-free basis of monomials in $H^\bullet(\overline{\calM}_{0,n+1},\mathbb{Z})$. There are two known descriptions of square-free bases of cohomology: the bases of Kontsevich and Manin \cite{KM} obtained from the Keel presentation, and the geometric bases of Gaiffi~\cite{Gai} that emerge in the context of the De Concini--Procesi presentation. Proposition \ref{prop:KeelNotPBW} implies that neither of those square-free bases can correspond to a Gr\"obner basis. As we shall see now, the Etingof--Henriques--Kamnitzer--Rains--Singh presentation is the ``right'' one for our purposes. One added bonus of that presentation (noticed already in \cite[p.~764]{EHKR}) is that it has an additional grading by the monoid $(2^{\underline{n}}, \cup)$ of all subsets of $\underline{n}$ with respect to taking unions, or even by the partition lattice $(\Pi_{\underline{n}},\vee)$ of all partitions of $\underline{n}$ ordered by reverse refinement; this extra grading is useful for various purposes, see e.g. the proof of Theorem~\ref{th:RHT}.

\begin{theorem}\label{th:Koszul}
The ring $H^\bullet(\overline{\calM}_{0,n+1},\mathbb{Z})$ is Koszul.
\end{theorem}

\begin{proof}
The proof of this result is noticeably long, so we shall begin with a short summary. According to Proposition \ref{prop:GBKoszul}, it is enough to demonstrate that the ring $\mathbf{R}_n$ has a quadratic Gr\"obner basis for a certain admissible ordering of generators. Thus, our first step is to describe a class of admissible orderings that are suitable for that purpose. After that, one has to justify that for such orderings, the ring $\mathbf{R}_n$ has a quadratic Gr\"obner basis. The proof of that claim goes in two steps. First, we shall perform Gaussian elimination to replace the defining relations of that ring by a set of relations with distinct leading monomials admitting a nice description in terms of combinatorics of finite ordered sets. Recall that normal monomials with respect to any set of generators of an ideal give an upper bound on the size of the quotient. In the case of the set of relations obtained from the defining relations of $\mathbf{R}_n$ via Gaussian elimination, we shall establish that the corresponding upper bound is sharp by mapping the basis $\mathsf{B}$ of $H_{\bullet}(\overline{\calM}_{0,n+1},\mathbb{Z})$ defined in Proposition \ref{prop:BasisHyperCom} onto the set of normal monomials surjectively, so the result of Proposition \ref{prop:GBCrit} will apply, completing the proof. 

For many aspects of our proof, it is convenient to define the ring $\mathbf{R}_n$ as a ``linear species'', i.e. to assign a similarly defined ring to any finite totally ordered set $A$ (and not just $A=\underline{n}$); we denote the corresponding ring by~$\mathbf{R}_A$. Throughout  the proof, we shall frequently use, for a non-empty finite ordered set $S$, the operation that removes from $S$ its maximal element; we denote this operation by~$\partial(S)$.

\smallskip

\noindent
\emph{Step I: choice of an ordering. } Our definition of a class of suitable admissible orderings of monomials begins with the following lemma. 

\begin{lemma}\label{lm:ExoticOrder}
Let $A$ be a finite ordered set. Consider the following binary relation $\prec'$ on the set $2^A$ of all subsets of $A$: we say that $I\prec' J$ if either $J=I\setminus \{a\}$ where $a\in I$, $a\ne\max(I)$, or $I=\partial(J)$. Let $\prec$ be the transitive closure of $\prec'$. Then $\prec$ is a partial order.  
\end{lemma}

\begin{proof}
It suffices to show that $\prec$ is antisymmetric. Assume the contrary, and take two subsets $I,J\subseteq A$ for which $I\prec J$ and $J\prec I$. This means that $I$ can be obtained from $J$ by a sequence of steps, each either removing the maximal element or adding a non-maximal element, and $J$ can be obtained from $I$ in the same way. Under each of these operations, the maximal element of a set does not increase, so if we are able to start from the set~$I$ and return to it after a few operations, then the maximal element remains unchanged, so at each step we just add a non-maximal element. Clearly, these steps increase cardinality, which is a contradiction.\qed
\end{proof} 

Let us fix some extension of the partial order $\prec$ from Lemma~\ref{lm:ExoticOrder} to a total order of $2^{A}$; we denote that extension by $\triangleleft$. We may use this order to define a graded lexicographic ordering of the generators $X_I$ ($I\subseteq A$, $|I|\ge 3$) of the ring $\mathbf{R}_A$; we denote that ordering by the same symbol~$\triangleleft$. The rest of the proof is dedicated to showing that for each such ordering $\triangleleft$, the ideal $\mathbf{I}_A$ of defining relations of $\mathbf{R}_A$ has a quadratic Gr\"obner basis.

\smallskip

\noindent
\emph{Step II: Gaussian elimination on the defining relations. } It is easy to see that among the leading monomials of the defining relations of the ring $\mathbf{R}_A$ there are repetitions. Let us apply a version of the Gaussian elimination procedure to get rid of repetitions, paying particular attention to the set of leading monomials thus obtained. 

The first group of elements generating $\mathbf{I}_A$ consists of $X_S^2$ for $|S|=3$; these elements are themselves monomials, we shall call them monomials of type 1a.  

The second group of elements consists of $X_S(X_S-X_{S\setminus\{s\}})$, where $s\in S$. Note that we have $X_{\partial(S)}\triangleleft X_S$ and $X_S\triangleleft X_{S\setminus\{s\}}$ for $s\ne\max(S)$. Thus, the relations of this group have distinct leading monomials: $X_S^2$, which we shall call a monomial of type 1b, and all $X_{S\setminus\{s\}}X_S$ for $s\ne\max(S)$, which we shall call monomials of type 2. 

Finally, we consider the elements $(X_{S\cup T}-X_S)(X_{S\cup T}-X_T)$ of the third group. This will lead us to some slightly nontrivial combinatorics; to simplify the notation, we put $I:=S\cup T$. There are two possibilities to consider, $\max(I)\in S\cap T$ and $\max(I)\notin S\cap T$. 

If $\max(I)\in S\cap T$, then $X_I\triangleleft X_S$ and $X_I\triangleleft X_T$, so $X_SX_T$ is the leading monomial of $(X_I-X_S)(X_I-X_T)$; this element does not appear in any other relation. We shall call such elements monomials of type 3a. 

Suppose that $\max(I)\notin S\cap T$. Without loss of generality, $\max(S)=\max(I)$ but $\max(T)\ne\max(I)$. Let us note that if $U$ is the shortest initial interval of $I$ containing~$T$ and $T\ne U$, then we may write 
 \[
(X_I-X_S)(X_I-X_T)=(X_I-X_S)(X_I-X_U)+(X_I-X_S)(X_U-X_T),
 \]
so subtracting the relation corresponding to $S$ and $U$, we obtain a relation  
 \[
(X_I-X_S)(X_U-X_T) .
 \]
Note that we have $X_I\triangleleft X_S$ and $X_U\triangleleft X_T$, so that as above, $X_SX_T$ is the leading monomial of this relation. We shall call such elements monomials of type 3b. 

It remains to consider the remaining relations of the third group, that is the relations 
$(X_I-X_S)(X_I-X_U)$, where $U$ is an initial interval of $I$ such that $S\cap U\ne\emptyset$ and $S\cup U=I$ (different initial intervals like that arise for different $T$ above). For a given subset $S$, suppose that $\partial(I)$, \ldots, $\partial^{k}(I)$ are all the proper initial intervals of $I$ satisfying this property. We note that
 $
X_I\triangleleft X_S \text{ and  }X_{\partial^{k}(I)}\triangleleft \cdots \triangleleft X_\partial(I)\triangleleft X_{I}
 $.
It is clear that the linear span of the relations
$(X_I-X_S)(X_I-X_{\partial^m(I)})$ with $m=1,\ldots, k$  
is the same as the linear span of the relations
 \[
(X_I-X_S)(X_I-X_{\partial(I)}), (X_I-X_S)(X_{\partial(I)}-X_{\partial^{2}(I)}),\ldots, (X_I-X_S)(X_{\partial^{k-1}(I)}-X_{\partial^k(I)}) .
 \]
Those relations have distinct leading monomials 
$X_SX_I, X_SX_{\partial(I)}, \ldots, X_SX_{\partial^{k-1}(I)}$. 
If $|S|<n-1$, these monomials are different from all the leading monomials considered above; in this case, we shall call these elements monomials of type 4a. For $|S|=n-1$, $X_SX_I$ is a leading monomial of type 2. Subtracting the corresponding relation, we may replace $(X_I-X_S)(X_I-X_{\partial(I)})$ by $(X_I-X_S)X_{\partial(I)}$, then subtracting this relation we may replace $(X_I-X_S)(X_{\partial(I)}-X_{\partial^{2}(I)})$ by $(X_I-X_S)X_{\partial^{2}(I)}$, and continuing in a similar way, we may replace $(X_I-X_S)(X_{\partial^{k-1}(I)}-X_{\partial^k(I)})$ by $(X_I-X_S)X_{\partial^{k}(I)}$. The leading monomials of these are $X_SX_{\partial(I)}$, \ldots, $X_SX_{\partial^{k-1}(I)}$, $X_SX_{\partial^k(I)}$, which are now pairwise distinct and different from any other monomials listed above.  We shall call these elements monomials of type 4b. 

\smallskip 

We summarise this step of the proof in the following statement. 

\begin{lemma}
The ideal of relations of the ring $\mathbf{R}_n$ is generated by the following quadratic elements with pairwise distinct leading terms:
\begin{itemize}
\item[\textbullet] $X_S^2$, $|S|=3$ (type 1a),
\item[\textbullet] $X_S^2-X_{\partial(S)}X_S$, $|S|>3$ (type 1b),
\item[\textbullet] $X_{S\setminus\{s\}}X_S-X_S^2$, $|S|>3$, $\max(S)\ne s\in S$ (type 2),
\item[\textbullet] $(X_S-X_{S\cup T})(X_T-X_{S\cup T})$, $S\cap T\ne\varnothing$, $S\not\subset T$, $T\not\subset S$, $\max(S)=\max(T)$ (type 3a),
\item[\textbullet] $(X_S-X_{S\cup T})(X_T-X_{S\cup T})$, $S\cap T\ne\varnothing$, $S\not\subset T$, $T\not\subset S$, $\max(S)>\max(T)$, $T$ not an initial interval of $S\cup T$ (type 3b),
\item[\textbullet] $(X_S-X_{S\cup T})(X_{\partial^{p-1}(S\cup T)}-X_{\partial^{p}(S\cup T)})$, $|(S\cup T)\setminus S|>1$, $p\ge1$, $S\cap \partial^{p}(S\cup T)\ne\varnothing$, $S\not\subset \partial^{p}(S\cup T)$, $\partial^{p}(S\cup T)\not\subset S$ (type 4a), 
\item[\textbullet] $(X_S-X_{S\cup T})X_{\partial^{p}(S\cup T)}$, $|(S\cup T)\setminus S|=1$, $p\ge1$, $S\cap \partial^{p}(S\cup T)\ne\varnothing$, $S\not\subset \partial^{p}(S\cup T)$, $\partial^{p}(S\cup T)\not\subset S$ (type 4b).
\end{itemize}
\end{lemma}

The rest of the proof is dedicated to showing that these generators form a Gr\"obner basis.

\smallskip

\noindent
\emph{Step III: normal quadratic monomials. } Let us describe all quadratic monomials that are normal with respect to our modified set of relations, that is all quadratic monomials that do not occur among the leading terms of those relations. In this description, a new notion important throughout the proof emerges; we begin with giving it a proper name.

\begin{definition}[MI-complement]
Suppose that $S$ and $T$ are two proper subsets of a finite ordered set $A$; as always in this paper, we assume $|S|, |T|\ge 3$. We shall say that $T$ is an \textsl{MI-complement} (minimal interval complement) of $S$ if the following conditions hold simultaneously:
\begin{itemize}
\item[\textbullet] the intersection of $S$ and $T$ is not empty,
\item[\textbullet] $T$ is an initial interval of $S\cup T$ (that is, if $s\in S\cup T$ and $s\le\max(T)$, then $s\in T$),
\item[\textbullet] among all the initial intervals of $S\cup T$ of cardinality at least three satisying the above conditions, $T$ is the shortest one.  
\end{itemize}
We shall say that $T$ is an \emph{essential MI-complement} of $S$ if it is an MI-complement of $S$ and $|T\setminus S|>1$.
\end{definition}

The definition of an essential MI-complement is a key combinatorial definition of this paper, so we feel that it is important to give several examples.

\begin{example}
Let $A=\underline{6}=\{1,2,3,4,5,6\}$. The set $T=\{1,2,3,4,5\}$ is an essential MI-complement of $S=\{1,4,6\}$ because it is an initial interval of $S\cup T=\{1,2,3,4,5,6\}$, and $S\cup\partial(T)=\{1,2,3,4,6\}\ne S\cup T$. On the contrary, $T'=\{1,2,3,4\}$ is not an essential MI-complement of $S$ even though $T'$ is an initial interval of $S\cup T'=\{1,2,3,4,6\}$: the problem is that $S\cup\partial(T')=\{1,2,3,4,6\}=S\cup T'$, so there is a shorter initial interval that can be taken. The set $T''=\{1,2,4\}$ is an MI-complement of $S$ because it is an initial interval of $S\cup T''=\{1,2,4,6\}$, and it is of length three, so there is nothing shorter; however, this MI-complement is not essential because $T''\setminus S=\{2\}$ is a set of cardinality one. The set $U=\{1,2,4\}$ is an essential MI-complement of each of the sets $V'=\{1,5,6\}$, $V''=\{2,5,6\}$, and $V'''=\{4,5,6\}$. 
\end{example}

\begin{lemma}\label{lm:NormalWeight2}
Suppose that $S, T\subseteq\underline{n}$ and $|S|,|T|\ge 3$. A commutative quadratic monomial $X_SX_T$ is normal with respect to the modified set of generators of the ideal $\mathbf{I}_A$ if and only if $\max(S)\ne\max(T)$ and one of the following three conditions hold:
\begin{itemize}
\item[\textbullet] the subsets $S$ and $T$ are disjoint,
\item[\textbullet] the subsets $S$ and $T$ are comparable,  
\item[\textbullet] one of them is an essential MI-complement of the other. 
\end{itemize}
\end{lemma}

\begin{proof}
Suppose $X_SX_T$ is a normal monomial. Because of commutativity, we may assume $\max(S)\ge \max(T)$. All the monomials $X_SX_T$ with disjoint $S$ and $T$ are manifestly normal; of course, in this case $\max(S)\ne \max(T)$. The only monomials $X_SX_T$ with $T\subseteq S$ that are not normal are the monomials of type 1a and 1b (all squares), monomials of type 2 (all monomials with $\max(T)\in S$ and $|S\setminus T|=1$), and the monomials that appear first in the lists of the type 4a (all monomials with $\max(T)\in S$ and $|S\setminus T|>1$), so normality is equivalent to $\max(S)>\max(T)$. It remains to consider the case of monomials $X_SX_T$ with $S$ and $T$ neither disjoint nor comparable. Because of the leading monomials of type 3a, we cannot have $\max(S)=\max(T)$ in such a normal monomial. If $T$ is not an initial interval of $S\cup T$, we find the monomial $X_SX_T$ among the monomials of type 3b, so it is not normal. Finally, if $T$ is an initial interval of $I=S\cup T$, then for $|I\setminus S|>1$ the monomial $X_SX_T$ is among the monomials of type 4a unless $T$ is the shortest initial interval of $I$ satisfying $I=S\cup T$, and for $|I\setminus S|=1$ all the monomials $X_SX_T$ are among the monomials of type 4b. It remains to notice that $I\setminus S=(S\cup T)\setminus S=T\setminus S$ to complete the proof. \qed
\end{proof} 

\smallskip

\noindent
\emph{Step IV: normal monomials and shuffle trees. } We shall now describe a one-to-one correspondence between the elements of the shuffle tree basis $\mathsf{B}$ of the operad $\HyperCom$ and commutative monomials whose divisors of weight two are all among the monomials described in Lemma \ref{lm:NormalWeight2}; as it was already done for the ring $\mathbf{R}_n$, we consider the shuffle operad $\HyperCom$ as a linear species, so that inputs of an operation may be indexed by elements of a finite ordered set~$A$. 

We begin with a recipe of how to construct the commutative monomial $\Phi(\tau)$ corresponding a shuffle tree $\tau\in \mathsf{B}$.
If $\tau=\stackrel{a}{\mid}$, we put $\Phi(\tau)=1\in\mathbf{R}_{\{a\}}$. Further, we define a particular class of shuffle trees that will be useful in the proof. Recall that a shuffle tree is called a \textsl{right comb} if for each its internal vertex, all the inputs of that vertex except for possibly the rightmost one are leaves. A specialty of right combs that distinguishes them from any other shuffle trees is that each of them has exactly one possible leaf labelling that makes it a shuffle tree: the one where the labels exhibit a \emph{global} increase from the left to the right. We shall be interested in a particular kind of right combs, those where each internal vertex except for possibly the root has exactly two leaves. We shall denote such tree $\tau_{k,\ell}$, where $k\ge 2$ is the number of leaves of the root vertex and $l\ge 0$ is the number of internal vertices different from the root. For example, we have 
 \[
\tau_{4,3}=
\vcenter
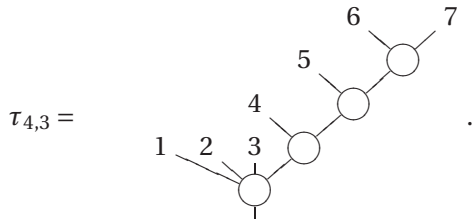
{
\xymatrix@R=1mm@C=2mm{
&&&&&&6\ar@{-}[dr]& &7\ar@{-}[dl]\\
&&&&&5\ar@{-}[dr]& &*++[o][F-]{}\ar@{-}[dl]\\
&&&&4\ar@{-}[dr]& &*++[o][F-]{}\ar@{-}[dl]\\
&&1\ar@{-}[drr]&2\ar@{-}[dr]& 3\ar@{-}[d]&*++[o][F-]{}\ar@{-}[dl]\\
&&&&*++[o][F-]{}\ar@{-}[d]&&\\
&&&&
}} .
 \]
By definition of the basis $\mathsf{B}$, for each element $\tau$ of that basis there exist unique numbers $k$ and $\ell$ as well as unique shuffle tree monomials $\tau_1$, \ldots, $\tau_k$, $\tau_{k+1}$, \ldots, $\tau_{k+\ell}$ each of which is either the trivial tree or a tree whose root vertex has strictly more than two children, such that $\tau$ is obtained by grafting the tree monomials $\tau_1$, \ldots, $\tau_k$, $\tau_{k+1}$, \ldots, $\tau_{k+\ell}$ at the leaves of~$\tau_{k,\ell}$. Let us put 
 \[
S_p=\Bigl\{ a\in\mathrm{Leaves}(\tau)\colon a\le \min(\mathrm{Leaves}(\tau_p))\Bigr\}. 
 \]
We define
 \[
\Phi(\tau)=\Phi(\tau_1)\cdots\Phi(\tau_{k+\ell}) \prod_{j=3}^{k}X_{S_{j+\ell}} 
 \] 
Here we use the convention according to which a product over the empty set is equal to one, so for $k=2$ we have $\Phi(\tau)=\Phi(\tau_1)\cdots\Phi(\tau_{2+\ell})$. Note that since we have the product starting from $j=3$, each set $S_j$ contains $\min(T_1)$, $\min(T_2)$ and $\min(T_3)$, so $|S_j|\ge 3$, and thus the monomial $\Phi(\tau)$ is indeed a monomial in our generators $X_S$. 

\begin{lemma}
For each $\tau\in\mathsf{B}$, all divisors of weight two of the monomial $\Phi(\tau)$ are among the monomials described in Lemma \ref{lm:NormalWeight2}. 
\end{lemma}

\begin{proof}
We prove this by induction on the arity of $\tau$. For arity one, there is nothing to prove. Suppose that that the statement is true for any arity less than $n$, and let $\tau$ be an element of $\mathsf{B}$ of arity $n$ obtained by grafting tree monomials $\tau_1$, \ldots, $\tau_k$, $\tau_{k+1}$, \ldots, $\tau_{k+\ell}$ at the leaves of $\tau_{k,\ell}$. Suppose that the given quadratic monomial $m=X_{S}X_{S'}$ is a divisor of $\Phi(\tau)$. If $m$ is a divisor of $\Phi(\tau_i)$ for some $i$, then it is normal by the induction hypothesis. If $m$ is a divisor of $\prod_{j=3}^{k}X_{S_{j+\ell}}$, then the subsets $S$ and $S'$ are comparable, as our product is defined as the product of some initial intervals of the set of leaves of $\tau$. If the generators $X_{S}$ and $X_{S'}$ are divisors of two different monomials $\Phi(\tau_i)$ and $\Phi(\tau_j)$, then $S$ and $S'$ are disjoint, since the monomial corresponding to any shuffle tree uses only the generators $X_U$ where $U$ is a subset of the set of leaves. Finally, suppose that $X_S$ is a divisor of $\Phi(\tau_i)$ for some $i$ and $S'=S_{j+\ell}$ for some $j\ge 3$. In this case, we may have the following situations:

\begin{itemize}
\item If $i>j+\ell$, then $\min(S)\ge\min(\mathrm{Leaves}(\tau_i))>\min(\mathrm{Leaves}(\tau_{j+\ell}))=\max(S')$, so $S$ and $S'$ are disjoint.
\item If $i=j+\ell$, there are two possibilities to consider.
\begin{itemize}
\item If $\min(S)>\min(\mathrm{Leaves}(\tau_i))=\max(S')$, then $S$ and $S'$ are disjoint.
\item If $\min(S)=\min(\mathrm{Leaves}(\tau_i))$, we remark that if $s<\max(S')$ and $s\in S$, then $s\in S'=S_{j+\ell}$ by the construction of $S_{j+\ell}$, so $S'$ is an initial interval of $S\cup S'$. Note that we have $\max(S')=\max(S_{j+\ell})=\min(\mathrm{Leaves}(\tau_i))=\min(S)<\max(S)$; moreover, this inequality shows that $\partial(S')$ is disjoint from $S$, so $S'$ is an MI-complement of $S$. This MI-complement is essential since $S'\setminus S$ contains as a subset the set 
$\{\min(\mathrm{Leaves}(\tau_1)),\ldots,\min(\mathrm{Leaves}(\tau_{2+\ell}))\}$
of cardinality at least $2$. 
\end{itemize}
\item If $i<j+\ell$, we start with remarking that since $\max(S)\in\mathrm{Leaves}(\tau_i)$ and $\max(S')=\min(\mathrm{Leaves}(\tau_{j+\ell}))\in\mathrm{Leaves}(\tau_{j+\ell})$, and the sets of leaves of $\tau_i$ and $\tau_{j+\ell}$ are disjoint, we have $\max(S')\notin S$. Now there are two possibilities to consider.
\begin{itemize}
\item If $\max(\mathrm{Leaves}(\tau_i))<\min(\mathrm{Leaves}(\tau_{j+\ell}))=\max(S')$, then $S\subset S'$, so since $\max(S')\notin S$, the monomial is normal.
\item If $\max(\mathrm{Leaves}(\tau_i))>\min(\mathrm{Leaves}(\tau_{j+\ell}))=\max(S')$, we remark that if $s<\max(S')$ and $s\in S$, then $s\in S'=S_{j+\ell}$ by the construction of $S_{j+\ell}$, so it is an initial interval of $S\cup S'$. If $S'=S\cup S'$, we have $S\subset S'$, so since $\max(S')\notin S$, the monomial is normal. If, on the other hand, $S'$ is a proper initial interval of $S\cup S'$, then the property $\max(S')\notin S$ implies that $S'$ is an MI-complement of $S$. This MI-complement is essential since the difference $S'\setminus S$ contains $\min(\mathrm{Leaves}(\tau_{j+\ell}))$ as well as at least one of the elements $\min(\mathrm{Leaves}(\tau_1))$ and $\min(\mathrm{Leaves}(\tau_2))$. 
\end{itemize}
\end{itemize}
This exhausts all possibilities, proving the required statement. \qed
\end{proof}

We should establish that $\Phi$ is a one-to-one correspondence. For that, we shall construct its inverse $\Psi$. The following definition will be useful for that.

\begin{definition}[Decomposable monomial]
Let us call a monomial $m\in \mathbf{R}_A$ decomposable if it can be written in the form $m=m'm''$, where $m'\in \mathbf{R}_{A'}$ and $m''\in \mathbf{R}_{A''}$ for two non-empty disjoint subsets $A',A''$ of~$A$. (In particular, if a monomial $m$ does not use a certain element $k\in\underline{n}$, it may be written as $m\cdot 1$, where $1\in\mathbf{R}_{\{k\}}$, from which it follows that it is decomposable.)
\end{definition}

Suppose that $m$ is a monomial all whose divisors of weight two are all among the monomials described in Lemma \ref{lm:NormalWeight2}. First, we assume that the monomial $m$ is decomposable, and write it as $m=m_1m_2\cdots m_k$, where all individual factors $m_i\in\mathbf{R}_{A_i}$ are not decomposable, and 
 \[ \min(A_1)<\min(A_2)<\cdots<\min(A_k).\] 
 We define 
 \[
\Psi(m)=
\vcenter
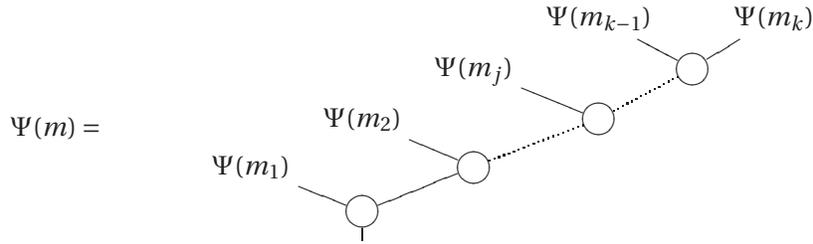
{
\xymatrix@R=1mm@C=2mm{
&&&&&&\Psi(m_{k-1})\ar@{-}[dr]& &\Psi(m_k)\ar@{-}[dl]\\
&&&&&\Psi(m_j)\ar@{-}[dr]& &*++[o][F-]{}\ar@{..}[dl]\\
&&&&\Psi(m_2)\ar@{-}[dr]& &*++[o][F-]{}\ar@{..}[dl]\\
&&&\Psi(m_1)\ar@{-}[dr]& &*++[o][F-]{}\ldots\ar@{-}[dl]\\
&&&&*++[o][F-]{}\ar@{-}[d]&&\\
&&&&
}}
 \]
It is clear that thus obtained tree monomial is normal. 

Assume now that $m\in \mathbf{R}_A$ is not decomposable. We shall construct the inverse by induction on weight of~$m$. Since $m$ is not decomposable, it is divisible by a generator $x_S$ with $\max(A)\in S$, and since $m$ is normal, it has exactly one such divisor.

Suppose first that $S=A$, so that $m=m'x_A$. The monomial $m'\in \mathbf{R}_{\partial(A)}$ is manifestly normal. Let us take the tree monomial $\Psi(m')$. As we saw before, this monomial is obtained by taking a certain right comb $\tau_{k,\ell}$ and grafting certain tree monomials $\tau_1$, \ldots, $\tau_k$, $\tau_{k+1}$, \ldots, $\tau_{k+\ell}$ at its inputs. To define the tree monomial $\Psi(m)$, we replace 
$\tau_{k,\ell}$ by $\tau_{k+1,\ell}$, graft the tree monomials $\tau_1$, \ldots, $\tau_k$, $\tau_{k+1}$, \ldots, $\tau_{k+\ell}$ at its first $k+\ell$ inputs, and make the last input a leaf labelled $\max(A)$. It is clear that thus obtained tree monomial is normal.

Suppose now that $S\subsetneq A$. In this case, for any other $x_T$ dividing $m$, we have either that $T\subset S$, or that $T$ and $S$ are disjoint, or that $T$ is an MI-complement of $S$. Let us factorise $m=m'm''$, where $m'$ is the product of all generators dividing $m$ that correspond to subsets of $S$. We now define a new monomial $f(m'')$ by keeping all generators $x_T$ dividing $m''$ for $T$ disjoint from $S$ unchanged, and replacing every generator $x_T$ for $T$ an MI-complement of $S$ by $x_{T\setminus (S\setminus \{\min(S)\})}$. Let us remark that for an MI-complement $T$ of $S$, we have $|T\setminus S|>1$, so $|T\setminus (S\setminus \{\min(S)\})|=|(T\setminus S)\cup \{\min(S)\}|>2$, and so all the generators are indexed by subsets containing at least three elements. 

\begin{lemma}
The monomial $f(m'')$ is normal.
\end{lemma}

\begin{proof}
We should establish that all divisors of $f(m'')$ of weight two are normal. We begin with remarking that the transformation of generators under $f$ does not change the properties of disjointness or inclusion of the corresponding subsets, thus we need to just consider what happens to two generators $x_{T_1}$ and $x_{T_2}$ where $T_1$ is an MI-complement of $T_2$. Note that the formula $x_T\leadsto x_{T\setminus (S\setminus \{\min(S)\})}$ is ``universal'': it also applies if $T$ is disjoint from $S$. 

The MI-complement condition implies in particular that we either have $|T_1|=3$ (then it is automatically the shortest possible) or $|T_1|>3$ and $\max(T_1)=\min(T_2)$ (in this case, $\partial(T_1)$ and $T_2$ are disjoint), or $|T_1|>3$ and $\max(T_1)\notin T_2$ (in this case, $\partial(T_1)\cup T_2\ne T_1\cup T_2$ so $T_1$ is the shortest initial interval complement). In the first of those cases, we already saw that $T_1$ remains unchanged under the given transformation. In the second case, $\max(T_1)=\min(T_2)$ remains unchanged also, since $\min(T_2)$ either does not belong to $S$ or is equal to $\min(S)$ (if $T_2$ is an MI-complement of $S$). In the third case, either $\max(T_1)$ does not belong to $S$, or it is equal to $\min(S)$ (if $T_1$ is an MI-complement of $S$), so it also remains unchanged under our transformation. This already shows that the results of transformation of $T_1$ and $T_2$ are neither disjoint nor comparable. The property of $T_1$ to be an initial interval of $T_1\cup T_2$ is clearly preserved under removing all elements of $S\setminus \{\min(S)\}$. The property stating that every element of $T_2$ not exceeding $\max(T_1)$ belongs to $T_1$ is also preserved under removing all elements of $S\setminus \{\min(S)\}$, as the maximum can only become smaller. Finally, as we already discussed above, the three possibilities to guarantee the minimality are all preserved too. \qed
\end{proof}

We define
 \[
\Psi(m)=\Psi(f(m''))\circ_{\min(S)}\Psi(m').
 \] 
Since $m'$ is divisible by $x_S$, it falls under the case that we considered, and we know that the root vertex of $\Psi(m')$ is non-binary, so the normality of $\Psi(f(m''))$ and $\Psi(m')$ implies normality of $\Psi(m)$. 

The last step of the proof is the following result that is crucial for us. 
\begin{lemma}
We have $\Phi(\Psi(m))=m$ for every normal monomial $m\in\mathbf{R}_A$. 
\end{lemma}

\begin{proof}
If $m$ is decomposable, then, in the notation above, we have 
 \[
\Phi(\Psi(m))=\Phi(\Psi(m_1))\cdots\Phi(\Psi(m_k))=m_1\cdots m_k=m  
 \] 
by the induction hypothesis. If $m$ is indecomposable and $m=m'x_A$, then the inductive definitions of the maps $\Psi$ and $\Phi$ show that 
$\Phi(\Psi(m))=\Phi(\Psi(m')) x_A=m'x_A=m$ by the induction hypothesis. Finally, if $m$ is indecomposable, and it is divisible by a generator $S$ such that $\max(A)\in S$, $S\ne A$, then, in the notation above,  
 \[
\Phi(\Psi(m))=\Phi(\Psi(f(m''))\circ_{\min(S)}\Psi(m')).
 \]
Note that since $m'$ is indecomposable, the right comb $\tau_{k,\ell}$ from the definition of the map $\Phi$ is the same for the tree $\Psi(f(m''))$ and the tree $\Psi(f(m''))\circ_{\min(S)}\Psi(m')$. Since the transformation $f$ does not change the minima of sets, it follows that all the MI-complements $S'$ of $S$ such that $S'\cup S=A$ will be restored correctly on the first step of the inductive definition of $\Phi$, and all the other MI-complements will be restored correctly by the inductive hypothesis.\qed
\end{proof}

To conclude the proof of the main result, we note that the last Lemma implies that the map $\Phi$ is surjective. Since the number of normal tree monomials is equal to the dimension of the arity $n$ component of the homology operad $H_\bullet(\overline{M}_{0,|A|+1},\mathbb{Q})$, which is the same as the rank of $\mathbf{R}_A$, this means that the number of normal monomials does not exceed that rank. The normal monomials with respect to any set of relations form a spanning set, so implies that the normal monomials must form a basis, and Proposition \ref{prop:GBCrit} ensures that the Gauss-reduced relations form a Gr\"obner basis. It follows from Proposition \ref{prop:GBKoszul} that the ring $\mathbf{R}_n$ is Koszul. \qed
\end{proof}

\section{Application to homotopy invariants of loop spaces}\label{sec:loop}

Let us discuss some applications of our result to the computation of homotopy invariants of loop spaces of moduli spaces of stable curves. 
Our main technical tool here is the framework of Koszul spaces developed over $\mathbb{Q}$ by Berglund~\cite{Ber} and extended for any field by Berglund and B\"orjeson~\cite{BBhighly}. Let us recall the main relevant result.

\begin{proposition}[{\cite[Theorem 2.15]{BBhighly}}]\label{prop:KoszulSp}
Let $\k$ be a field, and let $X$ be a connected space of finite $\k$-type. The following statements are equivalent:
\begin{enumerate}
\item The space $X$ is both formal and coformal over $\k$.
\item The space $X$ is formal over $\k$, and $H^\bullet(X,\k)$ is a Koszul algebra.
\item The space $X$ is coformal over $\k$, and $H_\bullet(\Omega X,\k)$ is a Koszul algebra.
\end{enumerate}
In such a situation, the associative algebras $H^\bullet(X,\k)$ and $H_\bullet(\Omega X,\k)$ are Koszul dual to each other. 
\end{proposition}

This result justifies the following definition.

\begin{definition}[Koszul space]
Let $\k$ be a field, and let $X$ be a connected space of finite $\k$-type. The space $X$ is said to be \emph{$\k$-Koszul} if either of the equivalent conditions of Proposition \ref{prop:KoszulSp} hold.
\end{definition}

As a first application of our main result, we shall use the theory of Koszul spaces to obtain a complete description of the rational homotopy Lie algebras and an estimate of growth of ranks of rational homotopy groups.

\begin{theorem}\label{th:RHT}
The rational homotopy Lie algebra $\pi_*(\Omega\overline{\calM}_{0,n+1})\otimes\mathbb{Q}$ is isomorphic to the graded Lie algebra generated by odd elements $Y_S$, where $S\subseteq\underline{n}$, $|S|\ge 3$, subject to relations 
\begin{gather*}
[Y_S,Y_T]=0, \quad \text{for each choice of $S$ and $T$ with } S\cap T=\varnothing, \\
\sum_{\substack{\{T_1,T_2\}\subset 2^S \colon\\ T_1\cap T_2\ne \varnothing, T_1\cup T_2=S}}[Y_{T_1},Y_{T_2}]=0, \quad \text{for } |S|>3, \\ 
\left[Y_T,\sum_{T\cup K=S} Y_K\right]=0, \quad \text{for each choice of $S$ and $T$ with } T\subset S, |S\setminus T|>1. 
\end{gather*}
This Lie algebra is finite dimensional for $n=2,3$, and has exponential growth for all $n\ge 4$. The generating function for ranks of rational homotopy groups is 
 \[
\sum_{m\ge 1}\frac{\mu(m)}{m}\ln\left(\frac{1}{f_n((-t)^m)}\right) , 
 \]
where $f_n(t^2)$ is the Poincar\'e polynomial for $\overline{\calM}_{0,n+1}$. 
\end{theorem}

\begin{proof}
Since $\overline{\calM}_{0,n+1}$ is a compact K\"ahler manifold, it is formal as a topological space~\cite{DGMS}. Thus, our main result implies that $\overline{\calM}_{0,n+1}$ is a Koszul space, and Proposition \ref{prop:KoszulSp} applies. By the theorem of Milnor and Moore \cite[p.~263]{MM}, the homology of the based loop space is the universal enveloping algebra of the rational homotopy Lie algebra:
 \[
H_\bullet(\Omega \overline{\calM}_{0,n+1},\mathbb{Q})\cong U(\pi_*(\Omega\overline{\calM}_{0,n+1})\otimes\mathbb{Q}) . 
 \]
Since $H_\bullet(\Omega \overline{\calM}_{0,n+1},\mathbb{Q})\cong H^\bullet(\overline{\calM}_{0,n+1},\mathbb{Q})^!$, we simply need to check that the relations listed here span the annihilator of relations of the algebra $\mathbf{R}_n\otimes\mathbb{Q}$. We shall first show that the relations listed in the statement of the theorem are orthogonal to the relations of $\mathbf{R}_n$. The first set of relations $[Y_S,Y_T]$ is clearly orthogonal to all relations of $\mathbf{R}_n$ because none of the relations of $\mathbf{R}_n$ include disjoint subsets as indices. The orthogonality property for the relation 
 \[
\sum_{\substack{\{T_1,T_2\}\subset 2^S \colon\\ T_1\cap T_2\ne \varnothing, T_1\cup T_2=S}}[Y_{T_1},Y_{T_2}] 
 \]
also follows by direct inspection of the relations of $\mathbf{R}_n$ (trivially orthogonal to the first group of relations, orthogonal to the second group by virtue of $1-1=0$, orthogonal to the last group by virtue of $1-1-1+1=0$). Finally, if we consider one of the relations $\left[Y_T,\sum_{T\cup K=S} Y_K\right]$, then it is trivially orthogonal to the first group of relations of $\mathbf{R}_n$, as well as all relations of the second group, since we have $|S\setminus T|>1$. It is also orthogonal to relations of the third group, since for each such relation $(X_{U\cup V}-X_U)(X_{U\cup V}-X_V)$, the pairing with $\left[Y_T,\sum_{T\cup K=S} Y_K\right]$ can only be nonzero if $T=U$ or $T=V$ (and $S=U\cup V$); in the former case, there are two matching terms, corresponding to $K=V$ and $K=U\cup V$, and the latter case is analogous. We also note that the first set of relations involves Lie monomials that do not appear in other relations, and so elements of that set cannot be used in a hypothetical linear dependency between the given relations; the relations of the first group are in one-to-one correspondence with normal monomials of the first type from Lemma~\ref{lm:NormalWeight2}. The element 
 \[
\sum_{\substack{\{T_1,T_2\}\subset 2^S \colon\\ T_1\cap T_2\ne \varnothing, T_1\cup T_2=S}}[Y_{T_1},Y_{T_2}] 
 \]
involves the Lie monomial $[Y_{\partial{S}}, Y_S]$ that does not appear in other relations, and so cannot be used in a hypothetical linear dependency between the given relations either. Finally, there are two types of elements of the third set of relations: those with $\max(S)\in T$ and those with $\max(S)\notin T$. A typical relation $\left[Y_T,\sum_{T\cup K=S} Y_K\right]$ of the first type ($\max(S)\in T$) involves the Lie monomial $[Y_T,Y_K]$, where $K$ is the MI-complement of $T$, and this monomial does not appear in any other relation we consider, so such elements of that set cannot be used in a hypothetical linear dependency between the given relations; such elements are in one-to-one correspondence with normal monomials of the third type from Lemma~\ref{lm:NormalWeight2}. A typical relation $\left[Y_T,\sum_{T\cup K=S} Y_K\right]$ of the second type ($\max(S)\notin T$) involves the Lie monomial $[Y_T,Y_S]$ that does not appear in any other relation of this type, and so such elements of that set cannot be used in a hypothetical linear dependency between the given relations either; such elements are in one-to-one correspondence with normal monomials of the second type from Lemma~\ref{lm:NormalWeight2} except for $X_{\partial(S)}X_S$ which is already accounted for. Thus, our relations are linearly independent, and their number is equal to the dimension of the weight two component of the Koszul dual algebra, so, being elements of the annihilator, they span it. This proves the first claim of the theorem. 

To establish the result on growth, we note that for $n=2$ the space $\overline{\calM}_{0,3}$ is a single point, so the rational homotopy Lie algebra is trivial, and for $n=3$ the space $\overline{\calM}_{0,4}$ coincides with $\mathbb{C}P^1$, so the rational homotopy Lie algebra is the free Lie algebra on one odd generator, which is finite-dimensional. To prove the dimension claim for $n\ge 4$, we use the grading of the Lie algebra $\pi_*(\Omega \overline{\calM}_{0,n+1})\otimes\mathbb{Q})$ by the monoid $(2^{\underline{n}}, \cup)$. This grading is useful since it shows that the obvious map
 \[
H_\bullet(\Omega \overline{\calM}_{0,4+1},\mathbb{Q})\to H_\bullet(\Omega \overline{\calM}_{0,n+1},\mathbb{Q})
 \] 
is injective: no further relations on the generators $Y_I$ with $I\subseteq\{1,2,3,4\}$ can follow from the other relations. Thus, it is enough to establish the dimension claim for $\overline{\calM}_{0,4+1}$. The algebra $H_\bullet(\Omega \overline{\calM}_{0,4+1},\mathbb{Q})$ has just one relation 
 \[ 
\sum_{\{T_1,T_2\}\subset 2^{\underline{4}}\colon T_1\cap T_2\ne \varnothing, T_1\cup T_2=\underline{4}}[Y_{T_1},Y_{T_2}] =0 ,
 \]
and for the ordering $Y_{123}>Y_{1234}>Y_{124}>Y_{134}>Y_{234}$, the leading term of this relation is $Y_{123}Y_{1234}$, which does not overlap itself nontrivially, so our algebra has a quadratic noncommutative Gr\"obner basis, and it is clear that the elements $Y_{124}$, $Y_{134}$, $Y_{234}$, and $Y_{1234}$ generate a free subalgebra. Thus, they generate a free Lie subalgebra in the Lie algebra $\pi_*(\Omega \overline{\calM}_{0,4+1})\otimes\mathbb{Q})$ which is therefore infinite-dimensional and has exponential growth.

For the last statement, we note that the generating series for dimensions of an associative algebra and and for dimensions of its Koszul dual are, up to alternating signs, multiplicative inverses of each other~\cite[Chapter 2, Corollary 2.2]{PP}. Thus, the generating function for dimensions of 
 \[
H_\bullet(\Omega \overline{\calM}_{0,n+1},\mathbb{Q})\cong H^\bullet(\overline{\calM}_{0,n+1},\mathbb{Q})^!
 \] 
is equal to $\frac{1}{f_n(-t)}$. It remains to use standard formula relating, for a simply connected space, the generating series for rational Betti numbers of the based loop space to the generating series of the ranks of rational homotopy groups, see~\cite{Bab80}.
\end{proof}

In principle, information about rational homotopy invariants of the based loop space $\Omega X$ can be used to derive information about such invariants for the free loop space $LX$, for instance estimate its Betti numbers. However, for a Koszul space, nothing new can be obtained here: such spaces are coformal, and by a theorem of Lambrechts~\cite{Lam}, for any simply connected coformal space $X$ of finite $\mathbb{Q}$-type, the rational Betti numbers of $LX$ grow exponentially whenever the rational homotopy Lie algebra of $X$ is infinite-dimensional. However, we are able to use our result for a similar conclusion in positive characteristic, strengthening in the particular case of moduli spaces of stable curves the results of~\cite{HV,McC}. 

\begin{theorem}\label{th:FreeLoopTorsion}
The sequence $\left\{\dim H_i(L\overline{\calM}_{0,n+1},\mathbb{F}_\ell)\right\}$ has exponential growth for all primes~$\ell\ge n$. 
\end{theorem}

\begin{proof}
First, we note that since the ring $H^\bullet(\overline{\calM}_{0,n+1},\mathbb{Z})$ has a quadratic Gr\"obner basis with leading coefficients equal to $1$, the algebra $H^\bullet(\overline{\calM}_{0,n+1},\mathbb{F}_\ell)$ has a quadratic Gr\"obner basis for all~$\ell$; therefore, that algebra is Koszul. 
To establish that $\overline{\calM}_{0,n+1}$ is $\mathbb{F}_\ell$-formal, we shall use the \'etale cohomology approach to formality developed by Cirici and Horel~\cite{CH}. Following their approach, we choose a prime number $p$ that generates $\mathbb{F}_\ell^\times$, take $K=\mathbb{Q}_p$, and note that $H^{2m}_{et}((\overline{\calM}_{0,n+1})_{\overline{K}},\mathbb{F}_l)$ is a Tate module that is pure of weight~$m$. Thus, by \cite[Theorem 8.2(iii)]{CH}, the space $\overline{\calM}_{0,n+1}$ is $2(p-2)$-formal over $\mathbb{F}_p$, which implies formality for $p\ge n$, since the complex dimension of $\overline{\calM}_{0,n+1}$ is equal to $n-2$. According to Proposition~\ref{prop:KoszulSp}, these two statements together imply that $\overline{\calM}_{0,n+1}$ is an $\mathbb{F}_\ell$-Koszul space whenever $\ell\ge n$. 

From results of Burghelea--Fiedorowicz and Goodwillie \cite{BF,GW}, for any pointed space $M$ and any unital commutative ring $\k$, we have
 \[
H_{\bullet}(LM,\k)\cong HH_\bullet(C_\bullet(\Omega M,\k),C_\bullet(\Omega M,\k)) .
 \]
Moreover, the homotopy invariance of Hochschild homology \cite[Prop.~III.2.9]{GW} ensures that for a $\k$-coformal space~$M$, we have
 \[
HH_\bullet(C_\bullet(\Omega M,\k),C_\bullet(\Omega M,\k))\cong HH_\bullet(H_\bullet(\Omega M,\k),H_\bullet(\Omega M,\k)) .
 \]
Thus, the commutator quotient
 \[
H_\bullet(\Omega M,\k)/[H_\bullet(\Omega M,\k),H_\bullet(\Omega M,\k)]\cong HH_0(H_\bullet(\Omega M,\k),H_\bullet(\Omega M,\k))
 \]
is a lower bound on the homology $H_{\bullet}(LM,\k)$ for any $\k$-coformal space~$M$. Applying this to the space $\overline{\calM}_{0,n+1}$ for $\ell\ge n$ which we have shown to be $\mathbb{F}_\ell$-Koszul, we conclude that the $\mathbb{F}_\ell$-Betti numbers of the free loop space $L\overline{\calM}_{0,n+1}$ grow exponentially whenever the dimensions of components of the commutator quotient 
 \[
H_\bullet(\Omega \overline{\calM}_{0,n+1},\mathbb{F}_\ell)/[H_\bullet(\Omega \overline{\calM}_{0,n+1},\mathbb{F}_\ell),H_\bullet(\Omega \overline{\calM}_{0,n+1},\mathbb{F}_\ell)]
 \]
of the algebra $H_\bullet(\Omega \overline{\calM}_{0,n+1},\mathbb{F}_\ell)$ grow exponentially. In particular, that is true whenever that algebra contains a free subalgebra on at least two generators, which is true in our case, by an argument identical to that in the proof of Theorem~\ref{th:RHT}. 
\end{proof}

\section{Generalisations}

Let us start with recording a rather obvious class of Koszul algebras which are superficially similar to the cohomology algebras we considered, but are much easier to analyse. We refer the reader to \cite{Dan,Ful} for background on polyhedral fans and toric varieties necessary for this section. 

\begin{theorem}\label{th:ToricKoszul}
Let $\Sigma$ be a smooth complete rational polyhedral fan. The cohomology ring $H^*(X_\Sigma,\mathbb{Z})$ of the complex toric variety $X_\Sigma$ is Koszul if and only if the fan~$\Sigma$ is a flag complex.  
\end{theorem}

\begin{proof}
By~\cite[Theorem 10.8]{Dan}, the cohomology algebra $H^*(X_\Sigma,\mathbb{Z})$ admits a presentation via generators corresponding to rays of the fan $\Sigma$ and relations of two types: linear relations and monomial relations of degree greater than one which are relations of the face ring of $\Sigma$. Those monomial relations are quadratic if and only if $\Sigma$ is a flag complex. Thus, the theorem essentially says that the cohomology of $X_\Sigma$ is Koszul if and only if it is quadratic. To prove this result, we recall that the proof of \cite[Theorem 10.8]{Dan} actually establishes that the basis of linear relations of the cohomology forms a regular sequence in the face ring of~$\Sigma$. It remains to note that the face ring of a flag complex is Koszul since it is a commutative algebra with monomial quadratic relations (which always form a Gr\"obner basis), and the quotient of a commutative Koszul ring by a regular sequence of linear forms is Koszul (over a field, one would use~\cite[Chapter 2, Cor.~5.4]{PP}; since we are over a ring, one has to replace the quotient by its linear resolution, the Koszul complex, and then pass to the total complex).
\end{proof}

Let us note that the cohomology rings of smooth projective toric varieties look very similar to the cohomology of $\overline{\calM}_{0,n+1}$: one takes a commutative algebra with monomial relations and quotients out a sequence of linear forms. However, for the case of $\overline{\calM}_{0,n+1}$, the number of linear forms grows quadratically in $n$, while the maximal length of a regular sequence grows linearly, so this result is gravely insufficient for the main theorem of this paper.

As a corollary to Theorem \ref{th:ToricKoszul}, we shall examine the noncommutative analogues $\mathrm{nc}\overline{\calM}_{0,n+1}$ defined in~\cite{DSV}. One of the geometric definitions of those varieties identifies them as (smooth projective) toric varieties whose fans are dual to Loday's realisations of associahedra, implying the following result.

\begin{corollary}
The ring $H^\bullet(\mathrm{nc}\overline{\calM}_{0,n+1},\mathbb{Z})$ is Koszul.
\end{corollary}

As another corollary to Theorem \ref{th:ToricKoszul}, we resolve the question of Koszulness of rational cohomology of Losev--Manin spaces $\overline{L}_n$ \cite{LM}. Those are known to be (smooth projective) toric varieties whose fans are dual to permutahedral polytopes; alternatively, they are the type $A$ fans associated to the Weyl chambers~\cite{Pro}.

\begin{corollary}
The ring $H^\bullet(\overline{L}_{n},\mathbb{Z})$ is Koszul.
\end{corollary}

The common generalisation of the spaces $\overline{\calM}_{0,n+1}$ and $\overline{L}_{n}$ is given by the genus zero components $\overline{L}_{0,S}$ of the ``extended modular operad''~\cite{LMExt}. All those components are particular cases of moduli spaces of rational weighted stable curves $\overline{\calM}_{0,\mathbf{w}}$ defined by Hassett~\cite{Has}. The cohomology algebras of those spaces are known to be quadratic, and Manin raised in \cite[Section 3.6.3]{Man} the question of Koszulness of those algebras. It is known that those spaces are wonderful models of certain hyperplane arrangements, see e.~g., \cite{CHMR,GR}. In particular, a presentation similar to the one we used in this paper is easy to obtain. We expect a slightly more technical version of the argument presented in this paper would confirm that those algebras are Koszul. In fact, we conjecture that 
that De Concini--Procesi wonderful models of hyperplane arrangements are Koszul in a wide range of cases. For an arbitrary arrangement and a chosen building set for the corresponding lattice, the cohomology algebra of the associated wonderful model is not always quadratic, but we suspect that this is the only obstacle for Koszulness. 

\begin{conjecture}\label{conj:DCP}
Consider a subspace arrangement in $\mathbb{P}(V)$ that refines a hyperplane arrangement. Let $\calG$ be a building set of the corresponding lattice of subspaces, and consider the De Concini--Procesi projective wonderful model $\overline{Y}_{\calG}$ associated to the building set~$\calG$. The ring $H^\bullet(\overline{Y}_{\calG},\mathbb{Z})$ is Koszul if and only if it is quadratic. 
\end{conjecture}

It would also be interesting to seek a further generalisation of our results to the case of algebras $D(\calL,\calG)$ defined by Feichtner and Yuzvinsky in \cite{FY} for a building set $\calG$ of an arbitrary atomistic lattice $\calL$. In general, those algebras admit geometric interpretation as Chow rings of certain smooth toric varieties, however those varieties are non-complete, so Theorem~\ref{th:ToricKoszul} is not applicable. We feel that to identify the class of lattices for which quadraticity of the algebra $D(\calL,\calG)$ implies its Koszulness, it may be useful to consider various notions of shellability of partially ordered sets~\cite{Bjo}, at least if one expects a quadratic Gr\"obner basis. We conclude with one more conjecture concerning a subclass of algebras $D(\calL,\calG)$ that are known to be quadratic: Chow rings of matroids~\cite{AHK}.

\begin{conjecture}
The Chow ring of any matroid is Koszul. 
\end{conjecture}

This conjecture would also automatically imply Koszulness of cohomology for the components of the extended modular operad. Indeed, in \cite{CHMR} those components are related to Bergman fans of graphic matroids; their cohomology is isomorphic to the appropriate matroid Chow rings. 

\providecommand{\bysame}{\leavevmode\hbox to3em{\hrulefill}\thinspace}

\end{document}